\documentclass[12pt]{amsart}

\usepackage{amsrefs}
\newtheorem{theorem}{Theorem}
\newtheorem{proposition}[theorem]{Proposition}
\newtheorem{lemma}[theorem]{Lemma}

\newtheorem{example}[theorem]{Example}
\newtheorem{remark}[theorem]{Remark}
\newtheorem{problem}[theorem]{Problem}

\begin{document}

\title{Conjugations on the Hardy space $H^{2}$}

\author[Marcos S. Ferreira]{Marcos S. Ferreira}

\address{%
Departamento de Ciências Exatas e Tecnológicas\\
Universidade Estadual de Santa Cruz\\
Ilhéus, Bahia, Brazil}
\email{msferreira@uesc.br}

\author[Geraldo de A. Júnior]{Geraldo de A. Júnior}

\address{%
Departamento de Ciências Exatas e Tecnológicas\\
Universidade Estadual de Santa Cruz\\
Ilhéus, Bahia, Brazil}
\email{gajunior@uesc.br}



\date{Dec 1, 2021}

\begin{abstract}
A conjugation $C$ on a separable complex Hilbert space $\mathcal H$ is an antilinear operator that is isometric and involutive. In this notes, we characterize all conjugations on the Hardy-Hilbert space $H^{2}$ over the disk. In addition, we characterize complex symmetric Toeplitz operators with a special type of these conjugations.
\end{abstract}

\maketitle

\section{Introduction}

A \textsl{conjugation} $C$ on a separable complex Hilbert space $\mathcal H$ is an antilinear operator $C:\mathcal H\rightarrow\mathcal H$ such that:
\begin{enumerate}
  \item [(a)] $C$ is \emph{isometric}: $\left\langle Cf,Cg\right\rangle=\left\langle g,f\right\rangle$, $\forall f,g\in\mathcal H$.
  \item [(b)] $C$ is \emph{involutive}: $C^{2}=I$.
\end{enumerate}

Conjugations have been widely studied recently and their roots are related to some fields of physics, especially quantum mechanics. The main motivation for the study of conjugations is the study of complex symmetric operators.

A bounded linear operator $T\in\mathcal{L}(\mathcal{H})$ is said to be \emph{complex symmetric} is there exists a conjugation $C$ on $\mathcal{H}$ such that $CT=T^{*}C$. In this case we say that $T$ is an operator \emph{$C$-symmetric}. The concept of complex symmetric operators on separable Hilbert spaces is a natural generalization of complex symmetric matrices, and their general study was initiated by Garcia, Putinar, and Wogen \cite{Garcia,Garcia2,Garcia3,Garcia4}. The class of complex symmetric operators includes other basic classes of operators such as normal, Hankel, compressed Toeplitz, and some Volterra operators.

Let $\mathbb{T}$ be the boundary of the open unit disk $\mathbb{D}$ in the complex plane $\mathbb{C}$. We let $L^{2}=L^{2}(\mathbb{T},\sigma)$ be usual Lebesgue space on $\mathbb{T}$ where $\sigma$ is the normalized Haar measure on $\mathbb{T}$. The \emph{Hardy space} $H^{2}$ consists of the all holomorphic functions $f$ on the unit disk $\mathbb{D}$ such that
$$
\sup_{0<r<1}\int_{\mathbb{T}}|f(r\zeta)|^{2}d\sigma(\zeta)<\infty.
$$

As is well known, the Hardy space $H^{2}$ is isometrically identified with a closed subspace of $L^{2}$ via the boundary functions. Indeed, if $f\in H^{2}$, considering the function $f_{r}$ given by $f_{r}(\zeta)=f(r\zeta)$, we have the radial limit
$$
f^{*}(\zeta)=\lim_{r\rightarrow1}f_{r}(\zeta)
$$
there exists for almost every $\zeta\in\mathbb{T}$ and holds
$$
\lim_{r\rightarrow1}\int_{\mathbb{T}}|f_{r}-f^{*}|^{2}d\sigma(\zeta)=0.
$$

Since $\left\{e_{n}(e^{i\theta})=e^{in\theta}: \ n\in\mathbb{Z}\right\}$ is an orthonormal basis for $L^{2}$, we have that $\left\{z^{n}:n=0,1,2,\ldots\right\}$ is an orthonormal basis for $H^{2}$ and therefore $f\in H^{2}$ if and only if
$$
f(z)=\sum_{n=0}^{\infty}a_{n}z^{n} \ \text{where} \ \sum_{n=0}^{\infty}|a_{n}|^{2}<\infty.
$$

For each $\phi \in L^{\infty }$, the \textit{Toeplitz operator} $T_{\phi}:H^{2}\rightarrow H^{2}$, with symbol $\phi$, is defined by%
\begin{equation*}
T_{\phi }f=P\left( \phi f\right),
\end{equation*}%
for all $f\in H^{2},$ where $P:L^{2}\rightarrow H^{2}$ is the orthogonal projection.

The concept of Toeplitz operators generalizes the concept of Toeplitz matrices and their general algebraic properties were studied by Brown and Halmos first addressed by \cite{Brown}. One of the first examples of complex symmetric Toeplitz operator is due to Guo and Zhu \cite{Guo}. In this work, Guo and Zhu raised the question of characterizing complex symmetric Toeplitz operators on the Hardy space $H^{2}$.

The most natural conjugation in $H^{2}$ is $\mathcal{J}$ given by
$$
\mathcal{J}f(z)=\overline{f(\overline{z})},
$$
or in general $C_{\lambda}f(z)=\overline{f(\lambda\overline{z})}$, with $\lambda\in\mathbb{T}$ (see \cite[Lemma 2.3]{Ko}). Thus, if we write $f(z)=\sum_{n=0}^{\infty}a_{n}z^{n}\in H^{2}$ we have to
\begin{eqnarray}\label{eq1}
C_{\lambda}f(z)=\sum_{n=0}^{\infty}\overline{a_{n}}\overline{\lambda^{n}}z^{n}.
\end{eqnarray}

Recently, a more general class of conjugations has been introduced in \cite{Li}. Li, Yang and Lu considered sequences $\alpha=\left\{\alpha_{0},\alpha_{1},\alpha_{2},\cdots\right\}$ where $\alpha_{m}\in\mathbb{T}$ and showed that $C_{\alpha}:H^{2}\rightarrow H^{2}$ defined by
\begin{eqnarray}\label{eq2}
C_{\alpha}\left(\sum_{n=0}^{\infty}a_{n}z^{n}\right)=\sum_{n=0}^{\infty}\overline{a_{n}}\alpha_{n}z^{n}
\end{eqnarray}
is a conjugation on $H^{2}$. Naturally, if $\alpha_{m}=\overline{\lambda^{m}}$ for each $m=0,1,2,\cdots$ then $C_{\alpha}=C_{\lambda}$.

The purpose of this paper is to use the fact that every conjugation $C$ on $H^{2}$ is of type $C=U^{*}\mathcal{J}U$, where $U$ is an unitary operator, and to find other conjugations on $H^{2}$ that generalize \eqref{eq2}. As a consequence, we guarantee a characterization for complex symmetric Toeplitz operators.

\section{Main results}

In \cite{Ferreira} Ferreira showed the important role that conjugation $\mathcal{J}f(z)=\overline{f(\overline{z})}$ plays in the study of conjugations on $H^{2}$, namely:

\begin{theorem}\label{teo1}
C is a conjugation on $H^{2}$ if, and only if, there exists an unitary operator $U:H^{2}\rightarrow H^{2}$ such that $C=U^{*}\mathcal{J}U$.
\end{theorem}
\begin{proof}
See Theorem 2.1 and Proposition 2.3 in \cite{Ferreira}.
\end{proof}

The previous theorem gives us a method of finding conjugations over $H^{2}$. Some are well known, such as \eqref{eq2}. The proof of the next lemma is left to the reader.

\begin{lemma}\label{lem1}
If $\left\{\zeta_{1},\zeta_{2},\ldots\right\}$ is a sequence where $\zeta_{j}\in\mathbb{T}$, then $\left\{1,\zeta_{1}z,\zeta_{2}^{2}z^{2},\ldots\right\}$ is an orthonormal basis for $H^{2}$.
\end{lemma}

\begin{proposition}\label{cor1}
If $\zeta=\left\{\zeta_{1},\zeta_{2},\ldots\right\}$ is a sequence where $\zeta_{j}\in\mathbb{T}$, then $\mathcal{C}_{\zeta}:H^{2}\rightarrow H^{2}$ defined by
\begin{eqnarray}\label{eq3}
\mathcal{C}_{\zeta}\left(\sum_{n=0}^{\infty}a_{n}z^{n}\right)=\sum_{n=0}^{\infty}\overline{a_{n}}\overline{\zeta_{n}^{2n}}z^{n}
\end{eqnarray}
is a conjugation on $H^{2}$.
\end{proposition}
\begin{proof}
Consider the unitary operator $U$ on $H^{2}$ given by $Uz^{n}=(\zeta_{n}z)^{n}$, with $n=0,1,2,\ldots$. We have from Theorem \ref{teo1} that $U^{*}\mathcal{J}U$ is a conjugation for $H^{2}$. We claim that $\mathcal{C}_{\zeta}=U^{*}\mathcal{J}U$. In fact, if $f(z)=\sum_{n=0}^{\infty}a_{n}z^{n}\in H^{2}$ then
\begin{eqnarray*}
  U^{*}\mathcal{J}Uf(z) &=& U^{*}\mathcal{J}(\sum_{n=0}^{\infty}a_{n}\zeta_{n}^{n}z^{n}) \\
   &=& U^{*}(\sum_{n=0}^{\infty}\overline{a_{n}}\overline{\zeta_{n}^{n}}z^{n}) \\
   &=& U^{*}(\sum_{n=0}^{\infty}\overline{a_{n}}\overline{\zeta_{n}^{2n}}(\zeta_{n}^{n}z^{n})) \\
   &=& \sum_{n=0}^{\infty}\overline{a_{n}}\overline{\zeta_{n}^{2n}}z^{n},
\end{eqnarray*}
as wished.
\end{proof}

Of course we can get \eqref{eq1} and \eqref{eq2} from \eqref{eq3}:

\begin{example}\label{exa1}
Let $e^{i\theta/2}\in\mathbb{T}$ and consider $\zeta=\left\{e^{i\theta/2},e^{i\theta/2},e^{i\theta/2},\ldots\right\}$. By Proposition \ref{cor1} we have $\mathcal{C}_{\zeta}$ is a conjugation on $H^{2}$ where
$$
\mathcal{C}_{\zeta}\left(\sum_{n=0}^{\infty}a_{n}z^{n}\right)=\sum_{n=0}^{\infty}\overline{a_{n}}\overline{(e^{i\theta/2})^{2n}}z^{n}=\sum_{n=0}^{\infty}\overline{a_{n}}\overline{(e^{i\theta})^{n}}z^{n}
$$
that is, $\mathcal{C}_{\zeta}=C_{e^{i\theta}}$.
\end{example}

\begin{example}\label{exa2}
Let $\zeta=\left\{\overline{e^{i\theta_{1}/2}},\overline{e^{i\theta_{2}/4}},\cdots,\overline{e^{i\theta_{n}/2n}},\cdots\right\}$. By Proposition \ref{cor1} $C_{\zeta}$ is given by
$$
\mathcal{C}_{\zeta}\left(\sum_{n=0}^{\infty}a_{n}z^{n}\right)=\sum_{n=0}^{\infty}\overline{a_{n}}\overline{\overline{(e^{i\theta_{n}/2n})}^{2n}}z^{n}=\sum_{n=0}^{\infty}\overline{a_{n}}e^{i\theta_{n}}z^{n}
$$
that is, $\mathcal{C}_{\zeta}$ is the conjugation \eqref{eq2}.
\end{example}

The standard in the coefficients of conjugations \eqref{eq1}, \eqref{eq2} and \eqref{eq3} is generally repeated. In fact, by Theorem \ref{teo1} all conjugations on $H^{2}$ are of the type $C=U^{*}\mathcal{J}U$ with $U$ unitary. Thus $C$ can be seen to be $A\mathcal{J}$ where $A=U^{*}\mathcal{J}U\mathcal{J}$ is an $\mathcal{J}$-symmetric unitary operator (see also \cite[Lemma 3.2]{Garcia5}). We have the following:

\begin{proposition}\label{cor2}
If $C$ is a conjugation on $H^{2}$, then
\begin{eqnarray}\label{eq4}
C\left(\sum_{n=0}^{\infty}a_{n}z^{n}\right)=\sum_{n=0}^{\infty}\sum_{m=0}^{\infty}\overline{a_{n}}b_{m}^{(n)}z^{m},
\end{eqnarray}
where $\left\{b_{m}^{(n)}\right\}_{m=0}^{\infty}$ it is a summable square sequence.
\end{proposition}
\begin{proof}
Let $C=A\mathcal{J}$, where $A$ is an $\mathcal{J}$-symmetric unitary operator. Thus, for all nonnegative integer $n$ we have
$$
A(z^{n})=\sum_{m=0}^{\infty}b_{m}^{(n)}z^{m}
$$
therefore
$$
C\left(\sum_{n=0}^{\infty}a_{n}z^{n}\right)=A\left(\sum_{n=0}^{\infty}\overline{a_{n}}z^{n}\right)=\sum_{n=0}^{\infty}\sum_{m=0}^{\infty}\overline{a_{n}}b_{m}^{(n)}z^{m},
$$
as wished.
\end{proof}

A natural question is: if $C$ is $C_{\lambda}$ or $C_{\alpha}$, which respective $b_{m}^{(n)}$ coefficients satisfy \eqref{eq4} in Proposition \ref{cor2}:

\begin{remark}
Let $C=C_{e^{i\theta}}$. By propositions \ref{cor1} and \ref{cor2} we have to $C_{\lambda}=A\mathcal{J}$, where $A=U^{*}\mathcal{J}U\mathcal{J}$ and $Uz^{n}=(e^{i\theta/2}z)^{n}$ for $n=0,1,2,\ldots$. Thus
$$
Az^{n}=U^{*}\mathcal{J}Uz^{n}=U^{*}\mathcal{J}(e^{i\theta/2}z)^{n}=e^{-in\theta/2}U^{*}z^{n}=\overline{(e^{i\theta})^{n}}z^{n}
$$
since $U^{*}z^{n}=e^{-in\theta/2}z^{n}$ and therefore
$$
\sum_{m=0}^{\infty}b_{m}^{(n)}z^{m}=\overline{(e^{i\theta})^{n}}z^{n},
$$
as wished. Analogously if $C=C_{\zeta}$, with $\zeta=\left\{\overline{e^{i\theta_{1}/2}},\overline{e^{i\theta_{2}/4}},\cdots,\overline{e^{i\theta_{n}/2n}},\cdots\right\}$, we have $Uz^{n}=(\overline{e^{i\theta_{n}/2n}}z)^{n}$ for $n=0,1,2,\ldots$ and therefore $Az^{n}=e^{i\theta_{n}}z^{n}$.
\end{remark}

Ko and Lee considered the family of conjugations $C_{\lambda}$ defined in \eqref{eq1} and proved the following:

\begin{theorem}\cite[Theorem 2.4]{Ko}
If $\varphi(z)=\sum_{n=-\infty}^{\infty}\widehat{\varphi}(n)z^{n}\in L^{\infty}$, then $T_{\varphi}$ is $C_{\lambda}$-symmetric if and only if $\widehat{\varphi}(n)\lambda^{n}=\widehat{\varphi}(-n)$ for all $n\in\mathbb{Z}$.
\end{theorem}

Analogous arguments made by Ko and Lee lead us to the following.

\begin{theorem}
Let $\zeta=\left\{\zeta_{1},\zeta_{2},\ldots\right\}$ a sequence, where $\zeta_{j}\in\mathbb{T}$, and $\varphi(z)=\sum_{n=-\infty}^{\infty}\widehat{\varphi}(n)z^{n}\in L^{\infty}$. Then $T_{\varphi}$ is $\mathcal{C}_{\zeta}$-symmetric if, and only if, $\widehat{\varphi}(n)\zeta_{n}^{2n}=\widehat{\varphi}(-n)$ for all $n=0,1,2\ldots$.
\end{theorem}

In view of the previous theorem, it is natural to ask:

\begin{problem}
Under what conditions a Toeplitz operator $T_{\varphi}$, with $\varphi\in L^{\infty}$, is $C$-symmetric where $C$ is given by \eqref{eq4}?
\end{problem}

\end{document}